\def\timestamp{%
Time-stamp: <lelek-nonmetric.tex: Tuesday 22-12-2009 at 10:57:04 (cet)>}
\def\stripname Time-stamp: <#1 #2>{#2}
\edef\filedate{\expandafter\stripname\timestamp}

\documentclass[a4paper]%
              {amsart}

\usepackage{amssymb}
\usepackage[backrefs]{amsrefs}

\newcommand{\cl}{\operatorname{cl}}
\newcommand{\orpr}[2]{\langle#1,#2\rangle}
\newcommand{\preim}{^\gets}

\newcommand\calB{\mathcal{B}}
\newcommand\calC{\mathcal{C}}
\newcommand\calU{\mathcal{U}}
\newcommand\calV{\mathcal{V}}

\DeclareMathSymbol\restr2{AMSa}{"16}
\DeclareMathSymbol\Q0{AMSb}{`Q}
\DeclareMathSymbol\R0{AMSb}{`R}

\if01
\DeclareMathSymbol\diagmap\mathord{symbols}{"34}
\DeclareMathSymbol\codiagmap\mathord{symbols}{"35}
\else
\DeclareMathSymbol\diagmap\mathord{AMSa}{"4D}
\DeclareMathSymbol\codiagmap\mathord{AMSa}{"4F}
\fi

\newtheorem{theorem}{Theorem}[section]
\newtheorem{proposition}[theorem]{Proposition}
\newtheorem{lemma}[theorem]{Lemma}

\theoremstyle{remark}

\newtheorem{question}[theorem]{Question}

\begin{document}

\title{Lelek's problem is not a metric problem}

\dedicatory{To Ken Kunen}

\author[D. Barto\v{s}ov\'a]{Dana Barto\v{s}ov\'a}
\address{Faculty of Mathematics and Physics\\
        Charles University in Prague\\
        Ke Karlovu 3\\
        121~16~Prague~2\\
        Czech Republic}
\email{dana.bartosova@gmail.com}

\author[K. P. Hart]{Klaas Pieter Hart}
\address{Faculty of Electrical Engineering, Mathematics and Computer Science\\
         TU Delft\\
         Postbus 5031\\
         2600~GA {} Delft\\
         the Netherlands}
\address{Department of Mathematics and Statistics\\
         Miami University\\
         Oxford\\
         OH~45056}
\email{k.p.hart@tudelft.nl}
\urladdr{http://fa.its.tudelft.nl/\~{}hart}

\author[L. C. Hoehn]{Logan C. Hoehn}
\address{Department of Mathematics\\ 
         University of Toronto\\
         Room~6290\\
         40 St.~George Street\\
         Toronto\\ Ontario\\ Canada M5S~2E4}
\email{logan.hoehn@utoronto.ca}

\author[B. van der Steeg]{Berd van der Steeg}
\address{DELTA N.V.\\ 
         afdeling Portfolio Analyse\\
         Poelendaelesingel 10\\
         4335~JA {} Middelburg\\
         the Netherlands}
\email{BvanderSteeg@delta.nl}

\date{\filedate}

\begin{abstract}
We show that Lelek's problem on the chainability of continua with span zero
is not a metric problem:
from a non-metric counterexample one can construct a metric one.
\end{abstract}


\maketitle

\section*{Introduction}

The notion of span of a metric continuum was introduced by Lelek 
in~\cite{MR0179766}, where he showed that chainable continua have span~zero,
and in~\cite{L2} he asked whether continua with span~zero are chainable.
This has become one of the classic problems of Continuum Theory, 
see~\cite{ency-f-01} for a recent survey.

The purpose of this paper is not to solve Lelek's problem; our goal is more
modest: we show that a non-metrizable counterexample to the problem may be 
converted into a metrizable one.
This makes the tools of infinitary combinatorics available to those searching
for a counterexample.

Our proof makes use of methods from Model Theory, most notably the 
L\"owen\-heim-Skolem theorem.
Given a non-metric continuum one can use this theorem to obtain a metric
quotient that shares many properties with the original space.
Indeed, we shall prove that the quotient will be chainable if{}f the original 
space is and likewise for having span~zero.
The proof of one of the four implications is much more involved than
that of the others as it relies on Shelah's Ultrapower Isomorphism
theorem from~\cite{MR0297554}.
This suggests an obvious question that we shall discuss at the end of 
this paper.

Section~\ref{sec.prelim} contains some preliminaries.
We repeat the definitions of chainability and the various forms of span.
We also describe the results from Model Theory that will be used in the
proofs.
In section~\ref{sec.relect} we prove our main results and in 
section~\ref{sec.questions} we discuss some questions related to the proofs.

\section{Preliminaries}
\label{sec.prelim}

\subsection{Chainability and span}
\label{subsec.chain.span}

Let $X$ be a continuum, i.e., a connected compact Hausdorff space.
We say $X$ is \emph{chainable} if every finite open cover has a refinement
that is a chain, which means that it can be enumerated as 
$\langle V_i:i<n\rangle$ such that 
$V_i\cap V_j\neq\emptyset$ if{}f $|i-j|\le1$.

We shall deal with four kinds of span: span, semispan, surjective span,
and surjective semispan.
Each is defined, for a \emph{metric} continuum~$(X,d)$,
as the supremum of all $\epsilon\ge0$ for which there is
a subcontinuum~$Z$ of $X\times X$ with the property that
$d(x,y)\ge\epsilon$ for all $(x,y)\in Z$ and
\begin{itemize}
\item $\pi_1[Z]=\pi_2[Z]$, in the case of span;
\item $\pi_1[Z]\subseteq\pi_2[Z]$, in the case of semispan;
\item $\pi_1[Z]=\pi_2[Z]=X$, in the case of surjective span; or
\item $\pi_2[Z]=X$, in the case of surjective semispan.
\end{itemize}
Thus any one of the spans is equal to zero if \emph{every}
subcontinuum of~$X\times X$ with the corresponding property from the list
must intersect the diagonal $\Delta_X$ of~$X$.
This then yields four definitions of having span~zero for general continua.

There are relations between these four kinds of span~zero, corresponding
to the inclusion relations between the defining collections of subcontinua
of~$X\times X$; see~\cite{HartvanderSteeg} for a diagram
and also for a proof that chainability implies that all spans are zero.

The diagram in~\cite{HartvanderSteeg} also mentions (surjective) symmetric 
span, but, as reported in~\cite{MR722431}, the dyadic solenoid, which is
not chainable, has symmetric span~zero, so that symmetric span~zero does
not characterize chainability.
The reader will be able to check that having (surjective) symmetric span~zero
is also covered by our reflection results.

\subsection{Wallman representation}

In the construction of the metric quotient we employ the Wallman
representation of distributive lattices.

We start with a compact Hausdorff space~$X$ and consider its lattice of
closed sets~$2^X$.
Any sublattice, $L$, of~$2^X$ gives rise to a continuous image 
of~$X$: the space~$wL$ of ultrafilters on~$L$.
If $a\in L$ then $\bar a$ denotes $\{u\in wL:a\in u\}$; the family 
$\{\bar a:a\in L\}$ is used as a base for the closed sets in~$wL$.
In general this yields a $T_1$-space; the space~$wL$ is Hausdorff if{}f
$L$~is \emph{normal}, which means that disjoint elements of~$L$ can be
separated by disjoint open sets that are complements of members of~$L$.

In general, a lattice embedding $h:L\to K$ yields a continuous
onto map $wh:wK\to wL$, where $wh(u)$~is the unique ultrafilter
on~$L$ that contains~$\{a:h(a)\in u\}$ (this family is a prime filter), 
so that in our case we obtain a continuous onto map $q_L:X\to wL$.

It should be clear that $X$ is the Wallman space of $2^X$.
However, one space may correspond to many lattices.
Indeed, if $\calC$~ is a base for the closed sets of~$X$ that is 
closed under finite unions and intersections then $X=w\calC$.

The article~\cite{Aarts} gives a good introduction to Wallman representations.

\subsection{Elementarity}

To construct the metric quotient mentioned in the introduction we need a
special sublattice of~$2^X$, an \emph{elementary} sublattice.

In general a substructure~$A$ of some structure~$B$ 
(a group, a field, a lattice)
is said to be an \emph{elementary} substructure if every sentence in the
language for the structure, with parameters from~$A$, that is true in~$B$ 
is also true in~$A$.
A sentence is a formula without free variables and such a formula is true
in a structure if it holds with all its quantifiers bound by that structure.

As a quick example consider the subfield $\Q$ of $\R$: it is \emph{not} an 
elementary subfield because of the following sentence:
$$
(\exists x)(x^2=2)
$$
The parameter $2$ belongs to $\Q$; the sentence holds in~$\R$ but does not hold
in~$\Q$.
This example illustrates the source of the power of elementarity: because
all existential statements true in the larger structure must be 
true in the substructure this substructure is very rich.
In fact, an elementary subfield of~$\R$ must contain all real algebraic 
numbers and it is a non-trivial result that these numbers do in fact
form an elementary subfield of~$\R$.

By a straightforward closing-off argument one shows that every subset
of a structure can be expanded to an elementary substructure --- this is
the L\"owenheim-Skolem theorem 
\cite{HodgesShorterModelTheory}*{Corollary~3.1.4}.
In full it states that a subset, $C$, of a structure~$B$ can be expanded
to an elementary substructure~$A$ whose cardinality is at most
$\aleph_0\cdot|C|\cdot|\mathcal{L}|$, where $\mathcal{L}$ is the
language used to describe the structures.
In the case of lattices the language is countable: one needs $\land$, $\lor$
and~$=$ as well as logical symbols and (countably many) variables.
Thus every lattice has a countable elementary sublattice.

As we discuss in Section~\ref{sec.questions} the expressive power of
the language of lattices is not strong enough for our purposes;
therefore we consider structures for the language of Set Theory.
Any reasonably large set will do but usually one takes a `suitably large'
regular cardinal number~$\theta$ and considers the set $H(\theta)$ of
sets that are hereditarily of cardinality less than~$\theta$, which means
that they and their elements and their elements' elements and \dots{}
all have cardinality less than~$\theta$.
The advantage of these sets is that they satisfy all of the axioms of 
Set Theory, except possibly the power set axiom.

What will be particularly useful to us is that if $M$ is an elementary 
substructure of~$H(\theta)$ then $\omega$~is both an element and a subset 
of~$M$; this is because $\omega$ and each finite ordinal are uniquely 
defined in~$H(\theta)$ by a formula with just one free variable; therefore
they automatically belong to~$M$.
As a consequence of this every finite subset of~$M$ is an element of~$M$
and this will give us the extra flexibility that we need.

We refer to \cite{MR597342}*{Chapters~IV and~V} for information on the
sets~$H(\theta)$ and elementarity in the context of Set Theory.
Note that the language of Set Theory has even fewer non-logical symbols 
than that of lattice theory: $\in$ and $=$.
The lattice operations, $\cap$ and $\cup$, are derived from these.

\subsection{Ultrapowers and ultracopowers}
\label{subsec.ultra}

We shall be using ultrapowers of lattices so we need to fix some notation.
Let $L$ be a lattice;
given an ultrafilter~$u$ on a cardinal number~$\kappa$ we define the
ultrapower~$\prod_uL$ of~$L$ by~$u$ to be the quotient of~$L^\kappa$
by the equivalence relation~$\sim_u$ defined by
$\langle x_\alpha:\alpha<\kappa\rangle\sim_u
 \langle y_\alpha:\alpha<\kappa\rangle$ if{}f 
$\{\alpha:x_\alpha=y_\alpha\}\in u$.
We turn $\prod_uL$ into a lattice by defining the operations pointwise.
There is an obvious embedding $\diagmap:L \to \prod_uL$, 
the diagonal embedding, defined by sending an element~$a$ to the 
(class of the) sequence~$\langle a:\alpha<\kappa\rangle$.

Dual to this is the notion of ultracopower of a compact Hausdorff space~$X$
by an ultrafilter~$u$.
One can define it in two equivalent ways.
The first is as the Wallman representation of the ultrapower~$\prod_u2^X$
of the lattice~$2^X$ by~$u$.

The second is via the \v{C}ech-Stone compactification.
Consider the product $\kappa\times X$, 
where $\kappa$ carries the discrete topology, 
and the two projections $\pi_X:\kappa\times X\to X$ 
and $\pi_\kappa:\kappa\times X\to\kappa$.
These have extensions, $\beta\pi_X:\beta(\kappa\times X)\to X$ and
$\beta\pi_\kappa:\beta(\kappa\times X)\to\beta\kappa$ respectively.
The preimage $\beta\pi_\kappa\preim(u)$ is homeomorphic to the Wallman
representation of~$\prod_u2^X$.
This follows from the facts that 
\begin{enumerate}
\item $\beta(\kappa\times X)$ is the Wallman representation 
      of~$2^{\kappa\times X}$, which in turn is isomorphic
      to $(2^X)^\kappa$; and 
\item if $F$ and $G$ are closed subsets of $\kappa\times X$ then the  
      intersections $\cl_\beta F\cap\beta\pi_\kappa\preim(u)$ and 
      $\cl_\beta G\cap\beta\pi_\kappa\preim(u)$ are equal if{}f the set 
      of~$\alpha$s for which 
      $F\cap(\{\alpha\}\times X)= G\cap(\{\alpha\}\times X)$ belongs to~$u$.
\end{enumerate}
The topological viewpoint enables us to see easily that one may use any base,
$\calC$, for the closed sets that is closed under finite unions and finite 
intersections to construct the ultracopower.
Indeed, if $F$ and $G$ are closed and disjoint in~$\kappa\times X$ then
a compactness argument applied to $\{\alpha\}\times X$ for each~$\alpha$
will yield sequences $\langle B_\alpha:\alpha<\kappa\rangle$ and 
$\langle C_\alpha:\alpha<\kappa\rangle$ in~$\calC$ such that 
$B_\alpha\cap C_\alpha=\emptyset$ for all~$\alpha$, and 
$F\subseteq\bigcup_\alpha\{\alpha\}\times B_\alpha$ and
$G\subseteq\bigcup_\alpha\{\alpha\}\times C_\alpha$.

This then can be used to show that the dual to the inclusion map 
$\calC^\kappa\to(2^X)^\kappa$ is injective, so that 
$\beta(\kappa\times X)=w(\calC^\kappa)$, and, similarly, that the dual
to the inclusion map $\prod_u\calC\to\prod_u2^X$ is injective, which gives
us that $\beta\pi_\kappa\preim(u)$~is the Wallman representation 
of~$\prod_u\calC$.

We denote the ultracopower of~$X$ by~$u$ as $\coprod_uX$.
Also, if $\langle F_\alpha:\alpha<\kappa\rangle$~is a sequence of closed subsets
of~$X$ then we let $F_u$ be the intersection
of $\cl_\beta(\bigcup_\alpha\{\alpha\}\times F_\alpha)$ with~$\coprod_uX$;
in case $F_\alpha=F$ for all~$\alpha$ the set $F_u$ corresponds to the
image of~$F$ under the diagonal embedding into~$\prod_u2^X$.

The restriction of $\beta\pi_X$ to $\coprod_uX$ is induced by the
diagonal embedding~$\diagmap$, we shall denote it by~$\codiagmap$.

\section{Reflections}
\label{sec.relect}

We fix a continuum $X$, a suitably large cardinal~$\theta$ and a countable 
elementary substructure~$M$ of~$H(\theta)$ with $X\in M$;
as $\theta$~was taken large enough the entities $X\times X$, $2^X$ 
and $2^{X\times X}$ belong to~$M$ as well, by elementarity.
We let $L=M\cap 2^X$ and $K=M\cap2^{X\times X}$.
The family $\calB_L=\{wL\setminus F:F\in L\}$ is a base for the open 
sets of~$L$.

As $M$ is countable, so are $L$ and $K$.
Therefore $wL$ and $wK$ are compact \emph{metrizable} spaces.
We shall have proved our main result once we establish that
$wL$~is chainable iff $X$~is and that $wL$~has span zero iff $X$~does.

\subsection{Chainability}
\label{subsec.chain}

We first show that $X$ is chainable if and only if $wL$ is.
The forward implication is easiest to establish.

\begin{proposition}[\cite{vdS}*{Section~7.2}]
If $X$ is chainable then so is $wL$. \label{prop.chain.down}
\end{proposition}

\begin{proof}
Let $\calU$ be a finite open cover of~$wL$.
By compactness we can find a finite subfamily~$\calB$ of~$\calB_L$ that 
refines~$\calU$.
Because every finite subset of~$M$ belongs to~$M$ we have $\calB\in M$.
Now the formula that expresses `$\calC$ is a chain refinement of~$\calB$' 
--- with $\calC$ as its only free variable --- is satisfied by a member
of~$H(\theta)$ and hence by an element of~$M$.
The latter consists of members of~$\calB_L$ and is a finite 
chain refinement of~$\calB$, and hence of~$\calU$.
\end{proof}

The converse implication is slightly harder to establish; in the proof
we employ the notion of a a precise refinement.
A \emph{precise} refinement of a cover~$\calU$ is a refinement, 
$\{V_U:U\in\calU\}$, indexed by~$\calU$ such that~$V_U\subseteq U$
for all~$U$.

\begin{proposition}[\cite{vdS}*{Section~7.3}]
If $X$ is not chainable then neither is~$wL$. \label{prop.chain.up}
\end{proposition}

\begin{proof}
There is an open cover of~$X$ that does not have an open chain refinement.
This statement can be expressed by a formula, with parameters in~$M$, 
that is quite complicated: expressing that a cover does not have a chain 
refinement involves a quantification over all finite sequences of elements 
of~$2^X$.

By elementarity this formula holds in~$M$, so we can take an open cover, 
$\calU$, of~$X$ that belongs to~$M$ and that satisfies the formula
\emph{with all quantifiers restricted to~$M$}, which
means that $\calU$ has no chain refinements that consist of 
members of~$\calB_L$.

As $\calU$ is a subset of~$\calB_L$ it also forms an open cover of~$wL$.
We must show that $\calU$ does not have any open chain refinement at all.
Let $\calV$ be any finite open refinement of~$\calU$.
By normality we can find a closed cover~$\{F_V:V\in\calV\}$ of~$wL$ such that
$F_V\subseteq V$ for all~$V$.
By compactness we can find finite subfamilies $\calB_V$ of~$\calB_L$
such that $F_V\subseteq\bigcup\calB_V\subseteq V$ for all~$V$.
Then $\mathcal{W}=\{\bigcup\calB_V:V\in\calV\}$ is a refinement of~$\calU$ 
that consists of members of~$\calB_L$, hence it is not a chain refinement.
As $\mathcal{W}$~is a precise refinement of~$\calV$ the latter is not
a chain refinement of~$\calU$ either.
\end{proof}

\subsection{Products}

To establish that (non-)zero span is reflected we need to explore the
relationship between $wL\times wL$ and $wK$.

It is clear, by elementarity, that $K$ contains the families
$\{A\times X:A\in L\}$ and $\{X\times A:A\in L\}$.
We use $L'$ to denote the sublattice of~$K$ generated by these families.
We trust that the reader will recognize the formula implicit in 
the following proof.

\begin{lemma}
If $F$ and $G$ are elements of~$K$ with empty intersection then there
are $F'$ and~$G'$ in~$L'$ such that $F\subseteq F'$, $G\subseteq G'$ 
and $F'\cap G'=\emptyset$.  
\end{lemma}

\begin{proof}
By compactness there are finite families $\calU$ and $\calV$ of basic open 
sets such that $F\subseteq\bigcup\calU$, $G\subseteq\bigcup\calV$ and
$\cl\bigcup\calU\cap\cl\bigcup\calV=\emptyset$.
By elementarity, and because $F,G\in M$ there are in~$M$ two 
sequences $\bigl<\orpr{A_i}{B_i},i<n\bigr>$ and 
$\bigl<\orpr{C_j}{D_j},j<m\bigr>$ of pairs of closed sets such that
$F\subseteq\bigcup_{i<n}(A_i\times B_i)$, 
$G\subseteq\bigcup_{j<m}(C_j\times D_j)$ and
$\bigcup_{i<n}(A_i\times B_i)\cap\bigcup_{j<m}(C_j\times D_j) =\emptyset$.
The two unions belong to~$L'$ and are the sets $F'$ and $G'$ that we seek.
\end{proof}

This lemma implies that $wK=wL'$ in the sense that $u\mapsto u\cap L'$
is a homeomorphism between the two spaces.
Furthermore it should be clear that $L'$ serves as a lattice base for the
closed sets of~$wL\times wL$, so that $wL'=wL\times wL$.

We find that $wK=wL\times wL$ by means of a natural homeomorphism~$f$:
the diagonal of the two maps $p_1$ and $p_2$ from $wK$ to $wL$:
$p_1(u)=\{A\in L:A\times X\in u\}$
and $p_2(u)=\{A\in L:X\times A\in u\}$.

This implies that the product map $q_L\times q_L: X\times X\to wL\times wL$
can be factored as $f\circ q_K$; here $q_L:X\to wL$ and 
$q_K:X\times X\to wK$ are the maps dual to the inclusions 
$L\subseteq 2^X$ and $K\subseteq 2^{X\times X}$ respectively.
It also follows that $p_1$ and $p_2$ correspond to the projections
from $wL\times wL$ to~$wL$.

Where possible we will suppress mention of the map~$f$ and simply
identify $wK$ with $wL\times wL$; we also use $q_K$ in stead of $q_L\times q_L$.

\subsection{Reflecting non-zero span}

Using the above result on products we prove the first reflection result
on span.

\begin{proposition}[\cite{vdS}*{Section~7.4}]
In the span (of any kind) of $X$ is non-zero then the span (of the same kind)
of~$wL$ is non-zero too.
\end{proposition}

\begin{proof}
Because having non-zero span is an existential statement we immediately
apply elementarity to conclude that there is $Z\in M$ that is a subcontinuum  
of~$X\times X$, that is disjoint from the diagonal $\Delta_X$ of~$X$ and 
has the corresponding property from the list in 
subsection~\ref{subsec.chain.span}.

Since $Z$ and $\Delta_X$ belong to~$K$ their images under $q_K$ are disjoint
as well, so that $q_K[Z]$~is a continuum in $wL\times wL$ that 
is disjoint from~$\Delta_{wL}$.

Using the properties of the maps $q_L$ and $q_K$ derived above it follows 
that $q_K[Z]$ satisfies the same property as $Z$.
For example, if $\pi_1[Z]\subseteq\pi_2[Z]$
then 
$\pi_1\bigl[q_K[Z]\bigr]=
 q_L\bigl[\pi_1[Z]\bigr]\subseteq 
 q_L\bigl[\pi_2[Z]\bigr]=
 \pi_2\bigl[q_K[Z]\bigr]$.

Thus $wL$ inherits any kind of non-zero span that $X$ may have.
\end{proof}

\subsection{Reflecting span zero}

We now turn to showing that having span zero (of any kind) is reflected
down from~$X$ to~$wL$.
We do this by proving the contrapositive, i.e., that having 
non-zero span reflects upward from $wL$ to~$X$.

To this end we assume that $Z$ is a subcontinuum of $wL\times wL$ that does
not meet the diagonal~$\Delta_{wL}$ of~$wL$ and satisfies the property associated
to the type of span under consideration.
The obvious thing to do would be to find a continuum~$Z'$ in~$X\times X$
with the same property as~$Z$ and such that $Z=q_K[Z']$,
for then $Z'$~is a witness to $X$~having non-zero span of the same kind 
as~$wL$. 

The only way to obtain this $Z'$ seems to be via Shelah's Ultrapower theorem
from~\cite{MR0297554}, which says that if two structures, $A$ and $B$, for 
the same language are elementarily equivalent then there are a 
cardinal~$\kappa$ and an ultrafilter~$u$ on~$\kappa$ such that the ultrapowers 
of $A$ and $B$ by~$u$ are isomorphic.

It was noted by Gurevi\v{c} in~\cite{MR929393} that if $A$~is an elementary
substructure of~$B$ then the isomorphism $h:A_u\to B_u$ can be chosen
in such a way that the following diagram commutes
$$
\vbox{
 \halign{\hfil$#$\hfil& ${}#{}$& \hfil$#$\hfil\cr
          A & \buildrel e \over\longrightarrow & B\cr
\llap{$\diagmap$}\big\downarrow & & \big\downarrow\rlap{$\diagmap$}\cr
        A_u & \buildrel h \over\longrightarrow & B_u\cr
}}
$$
here $\diagmap$ is the diagonal embedding of a structure into its ultrapower
and $e$~is the elementary embedding of~$A$ into~$B$.
Inspection of the proof in~\cite{MR0297554} will reveal that one can start
its recursive construction with the identity map on the diagonal 
in~$A^\kappa$.

In \cite{MR1056373}*{Lemma~2.8} Bankston used this observation to show 
that if $e:A\to B$ is an elementary embedding of lattices then every
continuum in~$wA$ is the image, under the map dual to~$e$, 
of a continuum in~$wB$.
We shall use the proof of this result with a few extra twists to find
the desired continuum~$Z'$ in~$X\times X$.

We expand the language of lattices by adding three unary function symbols:
$p_1$, $p_2$ and~$i$.
In the case of the lattice $2^{X\times X}$ we interpret these
as follows:
\begin{itemize}
\item $p_1(F)=\pi_1[F]\times X$;
\item $p_2(F)=X\times\pi_2[F]$; and
\item $i(F)=\bigl\{\orpr xy:\orpr yx\in F\bigr\}$.
\end{itemize}
These interpretations belong to~$M$ so that $K$~is also an elementary
substructure of~$2^{X\times X}$ with respect to the extended language.

We apply Gurevi\v{c}'s remark to $K$ and $2^{X\times X}$ to obtain a 
cardinal~$\kappa$ and an ultrafilter~$u$ on~$\kappa$ such that
there is an isomorphism, with respect to the extended language,
$h:\prod_uK\to \prod_u2^{X\times X}$ for which $\diagmap\circ e=h\circ\diagmap$.
The dual, $wh$, of~$h$ is a homeomorphism between $\coprod_u(X\times X)$ and 
$\coprod_uwK$ for which the dual equality $q_K\circ\codiagmap=\codiagmap\circ wh$
holds.
By the remark at the end of Subsection~\ref{subsec.ultra} we know that
$\coprod_uwK$ is the Wallman representation of both~$\prod_uK$ 
and~$\prod_u2^{wK}$.

We consider the closed subset $Z_u$ of $\coprod_uwK$.
We know that $Z=\codiagmap[Z_u]$ and that $Z_u$ is a continuum,
so $Z^+=(wh)^{-1}[Z_u]$ is a continuum as well.
We let $Z'=\codiagmap[Z^+]$.
Then $Z'$ is a subcontinuum of~$X\times X$ and 
$$
q_K[Z']=q_K\bigl[\codiagmap[Z^+]\bigr]
         =\codiagmap\bigl[wh\bigl[(wh)^{-1}[Z_u]\bigr]\bigr]
         =\codiagmap[Z_u]=Z
$$
Thus far we have followed Bankston's argument; we now turn to showing
that $Z'$ has the same property as~$Z$.
Because $q_K[Z']=Z$ we know that $Z'$~is disjoint from~$\Delta_X$.
As to the mapping properties: we shall prove that 
$\pi_1[Z]\subseteq\pi_2[Z]$ implies $\pi_1[Z']\subseteq\pi_2[Z']$, 
leaving any obvious modifications for the other cases to the reader.

Let $K_Z=\{F\in K:Z\subseteq \bar F\}$.
Since $K$ is a base for the closed sets of $wK$ we know that 
$Z=\bigcap\{\bar F:F\in K_Z\}$.
Next we observe that for $F\in K_Z$ there is $G\in K_Z$ such 
that $G\subseteq F$
and  $\pi_1[G]\subseteq\pi_2[F]$.
Indeed, let $G=F\cap\pi_1\preim\bigl[\pi_2[F]\bigr]$, then
$G\in K_Z$ because $\pi_1[Z]\subseteq\pi_2[Z]$, and 
$\pi_1[G]\subseteq\pi_1[F]\cap\pi_2[F]$.
When we reformulate this in terms of our extended language we find that
for every $F\in K_Z$ there is $G\in K_Z$ such that $G\subseteq F$ and
$i(p_1(G))\subseteq p_2(F)$.

Even though $Z$ is not (necessarily) a member of $K$ this carries over to 
$\coprod_uwK$, \emph{because} $\prod_uK$ is a base for the
closed sets of~$\coprod_uwK$ and because for every element 
$\langle F_\alpha:\alpha<\kappa\rangle$ of~$K^\kappa$ such that
$Z\subseteq F_\alpha$ for all~$\alpha$ we can find 
$\langle G_\alpha:\alpha<\kappa\rangle$ such that
$Z\subseteq G_\alpha\subseteq F_\alpha$ and 
$i(p_1(G_\alpha))\subseteq p_2(F_\alpha)$ for all~$\alpha$.

Thus we find that 
$Z_u=\bigcap\{\bar F:F\in \prod_uK_Z\}$ and for every $F\in\prod_uK_Z$ 
there is $G\in\prod_uK_Z$ such that $G\subseteq F$ and 
$i(p_1(G))\subseteq p_2(F)$.

Now apply the homeomorphism $(wh)^{-1}$ (and the isomorphism~$h$)
to see that the same holds 
for~$Z^+$ and the family $h\bigl[\prod_uK_Z\bigr]$, the latter
is equal to $\{G\in\prod_u2^{X\times X}:Z^+\subseteq \bar G\}$.

Finally, let $z$ be a point outside~$\pi_2[Z']$; we show it is not 
in~$\pi_1[Z']$ either.
To begin, $Z'$ and $X\times\{z\}$ are disjoint.
By compactness we can find open sets $U$ and $V$ with disjoint closures
such that $z\in U$ and $Z'\subseteq X\times V$.
Let $P=X\times(X\setminus U)$ and $Q=X\times(X\setminus V)$.
Now $Q_u\subseteq \codiagmap\preim[Q]$, so that $Q_u\cap Z^+=\emptyset$;
but $P_u\cup Q_u=\coprod_u(X\times X)$, hence $Z^+\subseteq P_u$.
Hence there is $\langle R_\alpha:\alpha<\kappa\rangle$ in~$\prod_u2^{X\times X}$
such that $Z^+\subseteq R_u\subseteq P_u$ 
and $\pi_1[R_\alpha]\subseteq \pi_2[P]$ for all~$\alpha$.
It follows that 
$\pi_1[Z']\subseteq\cl\bigcup_\alpha\pi_1[R_\alpha]\subseteq\pi_2[P]$, so that
$z\notin\pi_1[Z']$.

\section{Remarks and Questions}
\label{sec.questions}

\subsection{Elementarity, I}
The reader will undoubtedly have reflected on the amount of machinery that we
brought to bear on the seemingly simple properties of chainability and
having span~zero.
One would expect that taking an elementary sublattice of $2^X$ would 
be enough.
In the case of chainability this is not the case.
The proofs of propositions~\ref{prop.chain.down} and~\ref{prop.chain.up}
show that chainability is what one would call a base-independent property:
a continuum is chainable if{}f some\slash every lattice-base satisfies
the chainability condition.
On the other hand, as shown in~\cite{HartvanderSteeg} no ultracopower 
$\coprod_u[0,1]$ of the unit interval by an ultrafilter on~$\omega$ is 
chainable.
Now $2^{[0,1]}$ is an elementary substructure of its corresponding ultrapower;
hence $[0,1]$ and $\coprod_u[0,1]$ have elementarily equivalent bases:
they satisfy the same first-order lattice-theoretic sentences.
Because one space is chainable and the other is not we conclude that 
chainability is not expressible by a first-order sentence in the language
of lattices.

This changes when we use the language of set theory;
chainability is first-order when expressed in this language:
for every finite set~$\calU$ that is an open cover there are a finite 
ordinal~$n$ and an indexed family $\langle V_i:i<n\rangle$ of open sets
such that \dots.
We needed the expressive power of set theory to be able to
take finite subsets of our lattice of unspecified cardinality.

The proofs on span relied on the equality $wK=wL\times wL$, which again
needed the availability of all possible finite subsets of the substructure.

\subsection{Elementarity, II}

The proof on reflection of span~zero used Shelah's ultrapower
isomorphism theorem to associate to a continuum in~$wK$ a continuum 
in~$X\times X$.
This raises an obvious question.

\begin{question}
Can one obtain the continuum $Z'$ and prove its properties 
by more elementary (pun intended) means?  
\end{question}

The reflection of surjective (semi)span zero can be established by elementary
means, though without actually exhibiting a continuum~$Z'$ as in the question 
above.

To see this for the case of surjective semispan let $Z$ be a subcontinuum 
of~$wL\times wL$ that is disjoint from the diagonal and find $Y\in K$ that 
contains~$Z$ and is also disjoint from the diagonal.
Back in~$X\times X$ the closed set~$Y$ has the property that none of its
components maps onto~$X$ under the map~$\pi_2$.
Let $C$ be such a component and take $x\in X\setminus\pi_2[C]$;
as $C\cap(X\times\{x\})=\emptyset$ there must be a relatively clopen subset~$D$
of~$Y$ that contains~$C$ and that is also disjoint from~$X\times\{x\}$.
This yields a finite partition of~$Y$ into closed sets, none of which maps
onto~$X$ under~$\pi_2$.
By elementarity there is such a partition in~$M$; since $Z$~must be a subset
of one of the pieces of this partition we find that~$\pi_2[Z]\neq wL$.

If the case of surjective span each piece, $D$, of the partition will satisfy
`$\pi_1[D]\neq X$ or $\pi_2[D]\neq X$', resulting in
`$\pi_1[Z]\neq wL$ or $\pi_2[Z]\neq wL$'.

\smallskip
Another question is related to the result in~\cite{HartvanderSteeg} that
no ultracopower of~$[0,1]$ by an ultrafilter on~$\omega$ has span~zero.

\begin{question}
Is having span~zero a base-independent property?
\end{question}

If it is base-independent then the formulation cannot be first-order.

\end{document}